\documentclass[a4paper,14pt]{article} %размер бумаги устанавливаем А4, шрифт 12пунктов
\usepackage[T2A]{fontenc}
\usepackage[utf8]{inputenc}
\usepackage[russian,english]{babel} %используем русский и английский языки с переносами
\usepackage{amssymb,amsfonts,amsmath,mathtext,enumerate,float,amsthm} %подключаем нужные пакеты расширений
\usepackage[unicode,colorlinks=true,citecolor=black,linkcolor=black,urlcolor=black]{hyperref}
\usepackage{indentfirst} % включить отступ у первого абзаца
\usepackage[dvips]{graphicx} %хотим вставлять рисунки?
\graphicspath{{illustr/}}%путь к рисункам
\usepackage{enumitem}

\makeatletter
\renewcommand{\@biblabel}[1]{#1.} % Заменяем библиографию с квадратных скобок на точку:
\makeatother %Смысл этих трёх строчек мне непонятен, но поверим "Запискам дебианщика"

\usepackage{geometry} % Меняем поля страницы.
\DeclareMathOperator{\supp}{supp}

\usepackage{cite}

\theoremstyle{plain}
\newtheorem{theorem}{Theorem}[section]

\newtheorem{lemma}[theorem]{Lemma}

\theoremstyle{definition}
\newtheorem{definition}[theorem]{Definition}

\usepackage{mathtools}
\mathtoolsset{showonlyrefs}

\begin{document}

%\renewcommand{\bibname}{Список цитированной литературы}
%\renewcommand\refname{\bibname}
% !!!
% The text starts here

\title{
	The subspace $c_0$ is not complemented in $ac_0$
}

\author{
	N.N. Avdeev
	\footnote{
		This work was carried out at Voronezh State University and supported by the Russian Science
		Foundation grant 19-11-00197.
	}
	\footnote{nickkolok@mail.ru, avdeev@math.vsu.ru}
}

\maketitle

MSC 46B45


\paragraph{Abstract.}
We prove that the subspace $c_0$ of sequences that converge to zero
is not complemented in the space $ac_0$ of sequences that almost converge to zero.
We proceed with applying the same approach to inclusion chain $c_0\subset A_0 \subset \ell_\infty$.

\section{Introduction}

Let $\ell_\infty$ be the space of all bounded sequences
endowed with the usual norm
\begin{equation}
	\|x\| = \sup_{n\in\mathbb{N}} x_n
	,
\end{equation}
where $\mathbb{N}$ denotes set of all positive integers,
and let $c_0$ be the subspace  of null-sequences.
The famous result of Phillips~\cite{phillips1940linear}
states that $c_0$ is not complemented in $\ell_\infty$:

\begin{theorem}[\cite{phillips1940linear}]
	\label{thm:phillips}
	There is no bounded linear operator $P: \ell_\infty \to c_0$ such that for every
	$x \in c_0$ the equality $Px =x$ holds.
\end{theorem}

Theorem~\ref{thm:phillips} gives the very first example of a non-complemented subspace.
Later, more such subspaces were found;
for further discussion we refer the reader to~\cite{lindenstrauss1979classical}.

The notion of almost convergence is a generalization of the notion of convergence.
Before we proceed to it, we need to introduce the concept of Banach limits,
which (expectably) provide a way to generalize the concept of the limit.

\begin{definition}
	A linear functional $B\in\ell_\infty^*$ is said to be a \emph{Banach limit} if
	\begin{enumerate}[label=(\roman*)]
		\item
			$B\geq0$, that is $Bx \geq 0$ for $x \geq 0$,
		\item
			$B\mathbb{I}=1$, where $\mathbb{I} =(1,1,\ldots)$,
		\item
			$B(Tx)=B(x)$ for $x\in \ell_\infty$,
			where $T$ is the shift operator $T(x_1,x_2,\ldots)=(x_2,x_3,\ldots)$.
	\end{enumerate}
\end{definition}
The set of all Banach limits is denoted by $\mathfrak{B}$.
The existence of Banach limits was announced by Mazur~\cite{Mazur} and proved by Banach~\cite{banach1993theorie}.

\begin{definition}
	A sequence $x\in \ell_\infty$ is said to be \emph{almost convergent} to $t\in \mathbb{R}$
	if for every $B\in\mathfrak{B}$ the equality
	\begin{equation}
		Bx = t
	\end{equation}
	holds.
\end{definition}
The set of all sequences that are almost convergent to $t$ is denoted by $ac_t$;
the space of all almost convergent sequences is denoted by $ac$.

\begin{theorem}[Lorentz,~\cite{lorentz1948contribution}]
	A sequence $x=(x_1,x_2,...)$ is almost convergent to $t\in\mathbb{R}$ iff
	\begin{equation}
		\label{eq:crit_Lorentz}
		\lim_{n\to\infty} \frac{1}{n} \sum_{k=m+1}^{m+n} x_k = t
	\end{equation}
	uniformly by $m\in\mathbb{N}$.
\end{theorem}

Alekhno gave an elegant proof~\cite[Theorem 8]{alekhno2006propertiesII}
that $ac_0$ is not complemented in $\ell_\infty$.
The proof is based on the original Phillips's proof of Theorem~\ref{thm:phillips}
and uses some lemmas from there.
For more details on Banach limits, we refer the reader to~\cite{sucheston1967banach,Eberlein,semenov2020geomBL}.

In Section~\ref{sec:c0_in_ac0} of the present article, we prove that $c_0$ is not complemented in $ac_0$.
So, all the three inclusions
\begin{equation}
	c_0 \subset \ell_\infty,
	\quad
	ac_0 \subset \ell_\infty,
	\quad\mbox{and}\quad
	c_0 \subset ac_0,
\end{equation}
are not complemented.
Our proof is inspired by Whitley's approach~\cite{whitley1968projecting}
and the discussion~\cite{mathSE_Phillips}.

In Sections~\ref{sec:A0_in_ell_infty} and~\ref{sec:c0_in_A0}
we proceed with applying the same approach to the inclusion chain $c_0 \subset A_0 \subset \ell_\infty$,
where
\begin{equation}
	A_0 = \{x\in\ell_\infty: \alpha(x) = 0\}
	,
\end{equation}
\begin{equation}
	\alpha(x) = \varlimsup_{i\to\infty} \max_{i<j\leqslant 2i} |x_i - x_j|
	.
\end{equation}

The function $\alpha(x)$, first introduced in~\cite{our-vzms-2018},
may be considered as a way to characterize \emph{how much non-Cauchy a given bounded sequence is}.
Note that the condition $\alpha(x) =0$ is weaker than the Cauchy condition,
as well as the Lorentz's criterion~\eqref{eq:crit_Lorentz} is obviously weaker than
than the classic definition of the limit.

It has turned out that the functional $\alpha(x)$ is closely related to the theory of Banach limits.
In particular, the equality $A_0\cap ac_0 = c_0$ holds~\cite{avdeev2019space}.
In~\cite{our-ped-2018-alpha-Tx} it is shown that $TA_0 = A_0$ despite the fact that for some $x\in\ell_\infty$
one has $\alpha(Tx)<\alpha(x)$.
For further investigation of $\alpha(x)$ we also refer the reader to~\cite{avdeev2021subsets}.

\section{The subspace $c_0$ is not complemented in $ac_0$}
\label{sec:c0_in_ac0}

The following lemma is a classical result, and we prove it for the sake of completeness.
\begin{lemma}
	\label{lem:uncountable_subsets_of_N_with_finite_intersections}
	There exists an uncountable family of subsets
	$\{S_i\}_{i\in I}$, $S_i \subset \mathbb{N}$,
	such that each $S_i$ is countable and for every $i\neq j$ the intersection $S_i \cap S_j$ is finite.
\end{lemma}

\begin{proof}
	Consider a bijection $\mathbb{N} \leftrightarrow \mathbb{Q}$
	and let $I = \mathbb{R}$.
	For $i\in I$ we set $S_i = \{q_n\}$,
	where $\{q_n\} \subset \mathbb{Q}$ is a sequence which converges monotonically to $i$.
\end{proof}

The following lemma is inspired by~\cite{mathSE_Phillips}.

\begin{lemma}
	\label{lem:c_0_not_complemented_in_ac_0}
	For each linear operator $Q: ac_0 \to ac_0$ such that $c_0\subseteq \ker Q$,
	there exists infinite subset $S \subset \mathbb{N}$ such that
	\begin{equation}
		\forall(x \in ac_0 : \supp x \subset S)[Qx = 0]
		.
	\end{equation}
\end{lemma}

\begin{proof}
	Notice first that for any infinite subset $S \subset \mathbb{N}$
	there exists $x\in ac_0\setminus c_0$ with $\supp x \subseteq S$.
	Indeed, we can find such $x$ with $0 < x \leq \chi_S$
	that contains infinitely many ones, and the distance between the ones
	is enough for Lorentz's criterion~\eqref{eq:crit_Lorentz} to be hold.

	Let $\{S_i\}_{i \in I}$ be a family of subsets of $\mathbb{N}$
	that satisfy the conditions of Lemma~\ref{lem:uncountable_subsets_of_N_with_finite_intersections}.
	Suppose to the contrary that
	\begin{equation}
		\forall(\mbox{infinite }S\subset\mathbb{N})\exists(x \in ac_0 : \supp x \subset S)[Qx \neq 0]
		.
	\end{equation}
	In particular,
	\begin{equation}
		\forall(i\in I)\exists(x_i \in ac_0 : \supp x_i \subset S_i)[Q(x_i) \neq 0]
		.
	\end{equation}

	Note that $x_i \notin c_0$ because $c_0\subseteq \ker Q$.
	Without loss of generality we can assume that $\|x_i\|=1$ for all $i \in I$.

	Consider $I_n = \{i \in I\,:\,(Qx_i)_n \neq 0\}$,
	then $I = \bigcup\limits_{n\in\mathbb{N}} I_n$.
	Thus, we can find $n$ such that $I_n$ is also uncountable
	(otherwise $I$ would be countable as a countable union of countable sets,
	which contradicts to conditions of Lemma~\ref{lem:uncountable_subsets_of_N_with_finite_intersections}).

	Сonsider now $I_{n,k} = \{i \in I_n\,:\,|(Qx_i)_n| \geq 1/k\}$,
	then $I_n = \bigcup\limits_{k\in\mathbb{N}} I_{n,k}$.
	Applying the same argument as above, one can easily see that the set $I_{n,k}$ is uncountable for some $k$.
	Let us choose such $I_{n,k}$ and proceed with it.

	So, we have an uncountable set $I_{n,k}$ and
	\begin{equation}
		\forall(i\in I_{n,k})\exists(x_i \in ac_0 : \supp x_i \subset S_i)\Bigl[\|x_i\|=1 \mbox{~~and~~} |(Qx_i)_n| \geq 1/k\Bigr]
		.
	\end{equation}

	Consider a finite set $J \subset I_{n,k}$ with $\#J>1$
	(here $\#J$ stands for the cardinality of the set $J$).
	Take
	\begin{equation}
		y = \sum_{j \in J} \operatorname{sign}{(Qx_j)_n} \cdot x_j
		.
	\end{equation}
	Since the intersection $S_i \cap S_j$ is finite for any $i \neq j$ and
	$\supp x_j \subset S_j$,
	the intersection $\bigcap\limits_{j\in J} \supp x_j$ is also finite.
	Hence, $y = f + z$,
	where $\supp f$ is finite and $\|z\| \leq 1$.

	On the other hand,
	\begin{equation}
		\label{eq:non_complemented_sum_cardinality}
		(Qy)_n = \sum_{j \in J}
		(\operatorname{sign}(Qx_j)_n)
		\cdot (Qx_j)_n \geq \frac{\# J}{k}
		.
	\end{equation}
	Note, that $f\in c_0$ and we have $Qf = 0$, because $c_0 \subseteq \ker Q$.
	Thus, $Qy = Q(f+z) = Qf + Qz = Qz$ and
	\begin{equation}
		\label{eq:norm_Q_estimate}
		\frac{\# J}{k} \leq (Qy)_n \leq \|Qy\| = \|Qz\| \leq \|Q\| \cdot \|z\| \leq \|Q\|
		.
	\end{equation}
	Due to~\eqref{eq:norm_Q_estimate}, we obtain $\# J \leq \|Q\| k$ for every $J\subset I_{n,k}$.
	This contradicts the fact that $I_{n,k}$ is uncountable,
	and we are done.
\end{proof}

\begin{theorem}
	The subspace $c_0$ is not complemented in $ac_0$.
\end{theorem}

\begin{proof}
	Suppose to the contrary that
	there exists a continuous projection $P: ac_0 \to c_0$.
	Applying Lemma~\ref{lem:c_0_not_complemented_in_ac_0} to $I-P$
	we can find infinite subset $S\subset\mathbb{N}$
	such that $\forall(x\in ac_0 : \supp x \subset S)[(I-P)x = 0]$
	(such $x \in ac_0 \setminus c_0$ exists even if $\chi_S \notin ac_0$).
	But then $Px\notin c_0$,
	which contradicts the fact that $P$ is a projection onto $c_0$.
\end{proof}

\section{The subspace $A_0$ is not complemented in $\ell_\infty$}
\label{sec:A0_in_ell_infty}

\begin{lemma}
	\label{lem:c_0_not_complemented}
	For each linear operator $Q: \ell_\infty \to \ell_{\infty}$ such that $c_0\subseteq \ker Q$,
	there exists infinite subset $S \subset \mathbb{N}$ such that
	\begin{equation}
		\forall(x \in \ell_\infty : \supp x \subset S)[Qx = 0]
		.
	\end{equation}
\end{lemma}

\begin{proof}
	Let $\{S_i\}_{i \in I}$ be a family of subsets of $\mathbb{N}$
	such that conditions of Lemma~\ref{lem:uncountable_subsets_of_N_with_finite_intersections} are hold.
	Suppose to the contrary that
	\begin{equation}
		\forall(\mbox{infinite }S\subset\mathbb{N})\exists(x \in \ell_\infty : \supp x \subset S)[Qx \neq 0]
		.
	\end{equation}
	In particular,
	\begin{equation}
		\forall(i\in I)\exists(x_i \in \ell_\infty : \supp x_i \subset S_i)[Q(x_i) \neq 0]
		.
	\end{equation}
	The rest of the proof is similar to that of Lemma~\ref{lem:c_0_not_complemented_in_ac_0}.
\end{proof}

\begin{theorem}
	The subspace $A_0$ is not complemented in $\ell_\infty$.
\end{theorem}

\begin{proof}
	Suppose to the contrary that
	there exists a continuous projection $P: \ell_\infty \to A_0$.
	Then $c_0\subset A_0\subset\ker (I-P)$.
	Applying Lemma~\ref{lem:c_0_not_complemented} to $I-P$
	we can find infinite subset $S\subset\mathbb{N}$
	such that $\forall(x\in\ell_\infty : \supp x \subset S)[(I-P)x = 0]$.
	Let $S' \subset S$ be such that $\chi_S$ contains
	infinite quantities of both ones and zeros.
	But then $\alpha(P\chi_{S'}) = \alpha(\chi_{S'}) = 1$ and $P\chi_{S'} \notin A_0$,
	which contradicts the fact  that $P$ is a projection onto $A_0$.
\end{proof}

\section{The subspace $c_0$ is not complemented in $A_0$}
\label{sec:c0_in_A0}

To prove the fact, we need some auxiliary constructions from~\cite{our-vzms-2018}.

Let us define linear operator $F:\ell_\infty \to \ell_\infty$ as following:
\begin{equation}
	\label{operator_F}
	(Fy)_k = y_{i+2}, \mbox{ for } 2^i < k \leq 2^i+1
\end{equation}

For example,
$$
	F(\{1,2,3,4,5,6, ...\}) = \{1,2,\,3,3,\,4,4,4,4,\,5,5,5,5,5,5,5,5,\,6...\}
$$

It is easy to see that the equality
\begin{equation}
	\label{eq:alpha_F}
	\alpha(Fx) = \varlimsup_{k\to\infty} |x_{k+1} - x_{k}|
\end{equation}
holds.

Let us also define linear operator $M:\ell_\infty \to \ell_\infty$ as following:
\begin{multline*}
	M(\omega_1,\omega_2,...)=\left(
		0, 1\omega_1,
		0, \frac{1}{2}\omega_2, 1\omega_2, \frac{1}{2}\omega_2,
		0, \frac{1}{3}\omega_3, \frac{2}{3}\omega_3, 1\omega_3, \frac{2}{3}\omega_3, \frac{1}{3}\omega_3,
		0, ...,
	\right. \\ \left.
		0, \frac{1}{p}\omega_p, \frac{2}{p}\omega_p, ..., \frac{p-1}{p}\omega_p, 1\omega_p,
			\frac{p-1}{p}\omega_p, ..., \frac{2}{p}\omega_p, \frac{1}{p}\omega_p,
		0, \frac{1}{p+1}\omega_{p+1}, ...
	\right)
	.
\end{multline*}
Note that due to~\eqref{eq:alpha_F} we have $FM: \ell_\infty \to A_0$.

\begin{lemma}
	\label{lem:c_0_not_complemented_in_A_0}
	For each linear operator $Q: A_0 \to A_0$ such that $c_0\subseteq \ker Q$,
	there exists infinite subset $S \subset \mathbb{N}$ such that
	\begin{equation}
		\forall(x \in A_0 : \supp x \subset S)[Qx = 0]
	\end{equation}
	and $x\in A_0\setminus c_0$ such that $\supp x \subseteq S$.
\end{lemma}

\begin{proof}
	Let $\{U_i\}_{i \in I}$ be a family of subsets of $\mathbb{N}$
	such that conditions of Lemma~\ref{lem:uncountable_subsets_of_N_with_finite_intersections} are hold.
	Let $\{S_i\}_{i \in I}$ be a family of subsets of $\mathbb{N}$
	defined by the equality $S_i = \supp FM\chi_{U_i}$.
	Obviously, the family of sets $\{S_i\}_{i \in I}$ also
	satisfies the conditions of Lemma~\ref{lem:uncountable_subsets_of_N_with_finite_intersections}.
	Moreover, for every $i\in I$ we have $x = FM\chi_{U_i} \in A_0\setminus c_0$.

	Suppose to the contrary that
	\begin{equation}
		\forall(\mbox{infinite }S\subset\mathbb{N})\exists(x \in A_0 : \supp x \subset S)[Qx \neq 0]
		.
	\end{equation}
	In particular,
	\begin{equation}
		\forall(i\in I)\exists(x_i \in A_0 : \supp x_i \subset S_i)[Q(x_i) \neq 0]
		.
	\end{equation}
	The rest of the proof is similar to that of Lemma~\ref{lem:c_0_not_complemented_in_ac_0}.

\end{proof}

\begin{theorem}
	The subspace $c_0$ is not complemented in $A_0$.
\end{theorem}

\begin{proof}
	Suppose to the contrary that
	there exists a continuous projection $P: A_0 \to c_0$.
	Applying Lemma~\ref{lem:c_0_not_complemented_in_A_0} to $I-P$
	we can find infinite subset $S\subset\mathbb{N}$
	such that $\forall(x\in A_0 : \supp x \subset S)[(I-P)x = 0]$,
	and $x\in A_0 \setminus c_0$ such that $\supp x \subseteq S$.
	But then $Px = x\notin c_0$,
	which contradicts  the fact  that $P$ is a projection onto $c_0$.
\end{proof}

\section{Acknowledgements}
Author thanks Dr. Prof. E.M. Semenov and Dr. A.S. Usachev for encouragement and discussions.

%\printbibliography
%\printbibitembibliography


\begin{thebibliography}{99}
%
\bibitem{phillips1940linear}
\emph{Phillips} \emph{R. S.} On linear transformations // Transactions of the
American Mathematical Society. — 1940. — Vol. 48, no. 3. — Pp. 516–541.
%
\bibitem{lindenstrauss1979classical}
\emph{Lindenstrauss J., Tzafriri L.}
 Classical Banach spaces II: function spaces. – 1979. – Vol. 97.

\bibitem{Mazur}
\emph{Mazur} \emph{S.} O metodach sumomalności // Ann. Soc. Polon.
Math. (Supplement). — 1929. — Pp. 102–107.
%
\bibitem{banach1993theorie}
\emph{Banach} \emph{S.} Théorie des opérations linéaires. — Sceaux :
Éditions Jacques Gabay, 1993. — Pp. iv+128. — ISBN 2-87647-148-5. — Reprint
of the 1932 original.
%
\bibitem{lorentz1948contribution}
\emph{Lorentz} \emph{G. G.} A contribution to the theory of divergent
sequences // Acta Mathematica. — 1948. — Vol. 80, no. 1. — Pp. 167–190. —
ISSN 0001-5962.
%
\bibitem{alekhno2006propertiesII}
\emph{Алехно} \emph{Е.} Некоторые специальные свойства функционалов
Мазура. II // Аналитические методы анализа и дифференциальных
уравнений. Vol. 2. — Труды Института математики НАН Беларуси, 2006. —
Pp. 17–23.
%
\bibitem{sucheston1967banach}
\emph{Sucheston} \emph{L.} Banach limits // Amer. Math. Monthly. —
1967. — Vol. 74. — Pp. 308–311. — ISSN 0002-9890.
%
\bibitem{Eberlein}
\emph{Eberlein} \emph{W. F.} Banach-{H}ausdorff limits // Proc. Amer.
Math. Soc. — 1950. — Vol. 1. — Pp. 662–665. — ISSN 0002-9939.
%
\bibitem{semenov2020geomBL}
\emph{Семёнов} \emph{Е. М.}, \emph{Сукочев} \emph{Ф. А.},
\emph{Усачев} \emph{А. С.} Геометрия банаховых пределов и их
приложения // Успехи математических наук. — 2020. — Vol. 75, 4 (454. —
Pp. 153–194.
%
\bibitem{whitley1968projecting}
\emph{Whitley} \emph{R.} Projecting $m$ onto $c_0$ // The American
Mathematical Monthly. — 1966. — Vol. 73, no. 3. — Pp. 285–286.
%
\bibitem{mathSE_Phillips}
Why doesn’t $c_0$ admit a complement in $\ell _\infty $? — URL: \url
{https://math.stackexchange.com/questions/2467426}
%
\bibitem{our-vzms-2018}
\emph{Авдеев} \emph{Н. Н.}, \emph{Семенов} \emph{Е. М.} Об
асимптотических свойствах оператора Чезаро // Воронежская Зимняя
Математическая школа С.Г. Крейна – 2018. Материалы Международной
конференции. Под ред. В.А. Костина. — 2018. — Pp. 107–109.
%
\bibitem{avdeev2019space}
\emph{Avdeev} \emph{N.} On the Space of Almost Convergent Sequences //
Mathematical Notes. — 2019. — Vol. 105, no. 3/4. — Pp. 464–468. — URL: \url
{https://www.scopus.com/record/display.uri?origin=inward&eid=2-s2.0-85065492318}
%
\bibitem{our-ped-2018-alpha-Tx}\emph{Авдеев} \emph{Н. Н.} О суперпозиции оператора сдвига и одной
функции на пространстве ограниченных последовательностей // Некоторые
вопросы анализа, алгебры, геометрии и математического образования. —
2018. — Pp. 20–21.
%
\bibitem{avdeev2021subsets}
\emph{Avdeev} \emph{N.} On Subsets of the Space of Bounded Sequences //
Mathematical Notes. — 2021. — Vol. 109, no. 1. — Pp. 150–154.
\end{thebibliography}
\end{document}